\documentclass{article}

\usepackage{amsthm}
\usepackage{epsfig}
\newtheorem{definition}{Definition}
\newtheorem{theorem}{Theorem}
\newtheorem{lemma}{Lemma}
\newtheorem{remark}{Remark}
\newtheorem{question}{Question}

\title{On NP-hard graph properties characterized by the spectrum}

\author{Omid Etesami\thanks{e-mail etesami@gmail.com}
\\
{\it\small IPM, Tehran, Iran}
\\[3pt]
and
\\[3pt]
Willem H. Haemers\thanks{e-mail haemers@uvt.nl}
\\
{\it\small Tilburg University, The Netherlands}
}
\date{}

\begin{document}
	
\maketitle

\begin{abstract}
Properties of graphs that can be characterized by the spectrum of the adjacency matrix of the graph have been studied
systematically recently.
Motivated by the complexity of these properties, we show that there are such properties for which testing whether a graph has that
property can be NP-hard
(or belong to other computational complexity classes consisting of even harder problems).
In addition, we discuss a possible spectral characterization of some well-known NP-hard properties.
In particular, for every integer $k\geq 6$ we construct a pair of $k$-regular cospectral graphs,
where one graph is Hamiltonian and the other one not.
\\[3pt]
\end{abstract}

\section{Spectral characterizations}
We consider properties of graphs which can be characterized from the spectrum of the adjacency matrix of the graph.
Examples of such properties are being bipartite, being a disjoint union of complete graphs, being regular,
and being regular with a given girth (length of the shortest cycle).
(For this and other used results on spectral characterizations we refer to \cite{DH} and \cite{BH}).
The eigenvalues of the adjacency matrix can be computed in polynomial time from the adjacency matrix,
and for all examples mentioned above the property can be obtained from the spectrum in polynomial time.
This observation motivates the following question.

\begin{question}\label{Q}
Does there exist a graph property that is computationally hard to check but that can be characterized by the spectrum?
\end{question}

An obvious candidate for such a property is the following: 	

\begin{definition}
A graph is said to be DS ({\lq}determined by its spectrum{\rq}) if its spectrum uniquely determines its isomorphism class.
\end{definition}

Several graphs have this property, and it is conjectured that almost all graphs are DS.
As mentioned above, being the disjoint union of complete graphs is characterized by the spectrum.
More precisely, a graph $G$ (with at least one edge) is the disjoint union of complete graphs if and only if $G$ has smallest
eigenvalue $-1$.
But we can say more.
If the smallest eigenvalue is $-1$, then every other eigenvalue $\lambda$ is a positive integer,
and corresponds to a complete graph of order $\lambda+1$.
Thus we have:

\begin{lemma}\label{C}
A disjoint union of complete graphs is DS.
\end{lemma}

The graph property {\lq}being DS{\rq} is by definition a property determined by the spectrum.
So if {\lq}being DS{\rq} is NP-hard we have an affirmative answer to the above question.
However, it is not known if this is the case.

\begin{remark}
Given a graph $G$ that is not DS, a witness for $G$ not being a DS is a graph $H$ that has the same spectrum as $G$ together with a
witness that $H$ is non-isomorphic with $G$.
Since graph non-isomorphism is in the complexity class AM (Arthur-Merlin protocol)
(see \cite{AB}),  being DS is also in the complexity class co-AM (complement of AM).
\end{remark}

In the next section we will see that the answer to Question~\ref{Q} is affirmative anyhow.
Motivated by this result we investigate some famous NP-hard graph properties on being characterized by the spectrum.
Since being regular with a given girth (shortest cycle) is spectrally characterized,
we thought that being regular with a given longest cycle could be a good candidate.
We show that this idea doesn't work by constructing infinitely many pairs of cospectral regular graphs,
where one is Hamiltonian and the other one not.

\section{Encoding data in the spectrum}

We show that we can encode $n$ bits of information inside the spectrum of a graph with $O(n^2)$ vertices algorithmically in a way that
the $n$ bits can be recovered from the spectrum.
This shows that any property of $n$ bits (within a computational complexity class) can be translated as a property on graphs (within a
similar computational complexity class) that is characterizable by the spectrum.
	
In the following, for concreteness we instantiate the above general approach by considering the property of an $m$-vertex graph
being 3-colorable (which can be seen as a property on the $n = {m \choose 2}$ bits that specify the adjacencies of the graph.

\begin{definition}
For a given integer $m$, consider an ordering of all pairs $(u, v): 1 \le u < v \le m$. For $1 \le i \le {m \choose 2}$, let $(u_i, v_i)$
be the $i$'th pair in the ordering.
Given a graph $H$ with $m$ vertices, create the graph $G = T(H)$ with vertex set $S_1 \cup S_2 \cup \cdots \cup S_{m \choose 2}$, where $S_i$
are disjoint vertex sets with $|S_i| = i + 1$.
Two vertices $x$ and $y$ ($x\neq y$) of $G$ are adjacent if $x\in S_i$ and $y\in S_i$ for some $i$ for which $(u_i, v_i)$ is an edge of $H$.
\end{definition}


\begin{definition}
We say a graph $G$ has property $P$ if $G = T(H)$ for some graph $H$ that is 3-colorable.
\end{definition}	
	
\begin{lemma}
The property $P$ defined above for a graph $G$ can be determined from the spectrum of $G$.
\end{lemma}

\begin{proof}
By Lemma~\ref{C}, it is possible to check using its spectrum whether $G$ is a disjoint union of cliques,
and if this is the case we also have the sizes of the cliques.
Thus given $G$, we can figure out whether it is of the form $T(H)$ for some $H$:
We just need to check that there is at most one clique of each size $k$ for $2 \le k \le {m \choose 2} + 1$,
and no clique of size ${m \choose 2} + 2$ or larger.
Furthermore, when $G = T(H)$ for some $H$, we can easily deduce the unique $H$ from the spectrum of $G$.
In particular, the spectrum of $G$ has all the information on whether $H$ is 3-colorable or not.
\end{proof}
	
\begin{lemma}
The property $P$ defined above for a graph $G$ is $NP$-hard to compute.
\end{lemma}

\begin{proof}
We reduce the 3-coloring problem to whether a graph $G$ has the property $P$:
Given a graph $H$, we create in polynomial time the graph $G = T(H)$.
Then $H$ is 3-colorable if and only if $G$ has property $P$.
\end{proof}	

Putting together the above two lemmas, we have:

\begin{theorem}\label{P}
There is a property $P$ of graphs that is characterizable by the spectrum, but is NP-hard to compute.
\end{theorem}

\begin{remark}
If instead of starting from the 3-colorability problem of graph $H$, we started with some problem that is undecidable (uncomputable),
then we would get a property that is characterizable by the spectrum, but is undecidable.
If we started with a problem that was PSPACE-complete, we would get a property of graphs that is also PSPACE-complete but
characterizable by the spectrum.
And we could do the same for many other complexity classes.
\end{remark}

\section{Cospectral graphs and NP-hard properties}\label{c}

For a number of graph properties in P
pairs of cospectral regular graph have been constructed where one has the property and the other one not.
See~\cite{BCH} for having a perfect matching, \cite{H} for vertex and edge connectivity, and \cite{HS} for the diameter.
The regularity requirement makes the construction more complicated, but also a lot more interesting.
The reason is that cospectrality for regular graphs with respect to the adjacency matrix $A$ implies cospectrality with
respect to several other kinds of matrices, such as the Laplacian matrix $L=D-A$
($D$ is the diagonal matrix with the vertex degrees on the diagonal),
the signless Laplacian $Q=D+A$, and the normalized Laplacian $L_N=D^{-\frac{1}{2}}AD^{-\frac{1}{2}}$.
In this section we construct such pairs of cospectral regular graphs for some famous NP-hard graph properties.


\subsection{Latin square graphs}

A good source for such pairs of graphs are the Latin square graphs defined as follows:
Given an $m\times m$ Latin square $L$, the vertices of the Latin square graph $G(L)$ are
the $m^2$ entries of $L$, two vertices being adjacent whenever the entries are in the same row, the same column,
or have the same symbol.
A Latin square graph $G(L)$ of order $m^2$ is a strongly regular graph with spectrum
$\{3m-3,(m-3)^{3m-3},(-3)^{m^2-3m+2}\}$
(exponents indicate multiplicities), so the spectrum only depends on $m$.
Therefore, two Latin squares of the same size give cospectral regular graphs.
A coclique (indepedent set of vertices) in $G(L)$ consists of a set of entries in $L$ all in different rows,
columns and with different symbols.
If such a set has cardinality $m$ (which is obviously the maximum) it is called a transversal.
It is known (see~\cite{W} for this, and other used facts about transversals in Latin squares)
that for even $m\geq 4$ there exist Latin squares with and without a transversal,
which shows that there are cospectral regular graphs with different independence number.
The chromatic number of a Latin square graph $G(L)$ is at least $m$, and equality holds if and only if $L$ has
$m$ disjoint transversals, which means that it has an orthogonal mate.
For $m\geq 4$, $m\neq 6$ there exists Latin squares with and without an orthogonal mate.
This shows that there exist cospectral Latin square graphs with different chromatic numbers.

Since Latin square graphs are regular the spectrum of the complement follows from the spectrum of $G(L)$.
So by considering the complements we see that also the clique number, as well as the clique covering number are not
characterized by the spectrum.

\subsection{Hamiltonian graphs}

We mentioned that being regular with a given girth 
is characterized by the spectrum.
(Without the regularity condition, the statement is false.)
So being regular and Hamiltonian (having a cycle of length equal to the order), which is a NP-hard property
could be a candidate for being characterized by the spectrum.
We will show that this is not the case.
Latin square graphs are regular and Hamiltonian, and so are their complements when $m\geq 4$.
Thus we need a different kind of cospectral graphs, which can be made by the following method.

\begin{theorem}\label{GMswitching}
Let $G$ be a graph and let $X$ be a subset of the vertex set of $G$ which induces a regular subgraph.
Assume that each vertex outside $X$ is adjacent to $|X|$, $\frac{1}{2}|X|$ or $0$ vertices of $X$.
Make a new graph $G'$ from $G$ as follows.
For each vertex $v$ outside $X$ with $\frac{1}{2}|X|$ neighbors in $X$,
delete the $\frac{1}{2}|X|$ edges between $v$ and $X$, and join $v$ instead to the $\frac{1}{2}|X|$
other vertices in $X$.
Then the adjacency matrices of $G$ and $G'$ have the same spectrum.
\end{theorem}

This result is due to Godsil and McKay, and the operation that changes $G$ into $G'$ is called
\emph{Godsil-McKay switching} (GM-switching).
The set $X$ will be called a \emph{(GM-)switching set}.

\begin{theorem}\label{H}
For every $k\geq 6$ there exists a pair of $k$-regular cospectral graphs, where one is Hamiltonian and the other one not.
\end{theorem}

\begin{proof}
We define the graph $H_k$, which can be obtained from the complete graph $K_{k+1}$ by the deletion of one edge
$\{a,b\}$ (say), and adding two pendant vertices at $a$ and $b$.
The following observation is trivial, but relevant.

\begin{lemma}\label{LH}
Suppose $\{c,d\}$ is an edge in a $k$-regular graph $G$,
and let $\hat{G}$ be the $k$-regular graph obtained from $G$ by replacing the edge $\{c,d\}$ by $H_k$ with pendent vertices $c$ and $d$.
Then $\hat{G}$  is Hamiltonian if and only if $G$ has a Hamiltonian cycle that contains $\{c,d\}$.
\end{lemma}

First we construct a $k$-regular graph $G$ with a GM-switching set $X$ of order $2(k-2)$.
Start with four disjoint cycles:
$C_1=(V_1,E_1)$ of length $k-3$, $C_2=(V_2,E_2)$ of length $k-1$,
and $C_3=(V_3,E_3)$ and $C_4=(V_4,E_4)$, both of length $k-2$.
Choose an edge $\{x,y\}\in E_2$, an edge $\{v,w\}\in E_3$, and a vertex $z\in V_4$.
Insert all possible edges between $V_1\cup\{x\}$ and $V_3$, and between  $V_2\setminus\{x\}$ and $V_4$.
Next, delete the edges $\{w,x\}$ and $\{y,z\}$ and insert $\{w,y\}$ and $\{x,z\}$.
The graph $G$, thus obtained is regular of degree $k$ and the set $X = V_1\cup V_2$ is a GM-switching set.
Let $G'$ be the graph obtained by switching.
If in $G'$ we interchange $C_3$ and $C_4$, and then reverse the order of the vertices in $C_3$ and $C_4$,
we get an exact copy of $G$, which shows that $G$ and $G'$ are isomorphic.
For $k=6$ this is illustrated in Figure~\ref{G} (names between parentheses belong to $G'$).
One straightforwardly checks that $G$ (and $G'$) is Hamiltonian.
However, no Hamiltonian cycle in $G$ contains all edges of $E_3\setminus \{v,w\}$.
Indeed, a Hamiltonian cycle in $G$ must contain the edge $\{w,y\}$, and can not be completed if it also contains the edges of
$E_3\setminus \{v,w\}$.
However, one easily finds a Hamiltonian cycle in $G'$ that contains the edges of $E_3\setminus \{v,w\}$.
Now, we make larger graphs $\hat{G}$ and $\hat{G'}$ by replacing each edge of $E_3\setminus \{v,w\}$ by $H_k$.
Then $X$ remains a switching set, so $\hat{G}$ and $\hat{G'}$ are still cospectral, but by Lemma~\ref{H},
$\hat{G}$ is not Hamiltonian and $\hat{G'}$ is.
\end{proof}

\begin{figure}[h]
\epsfig{file = 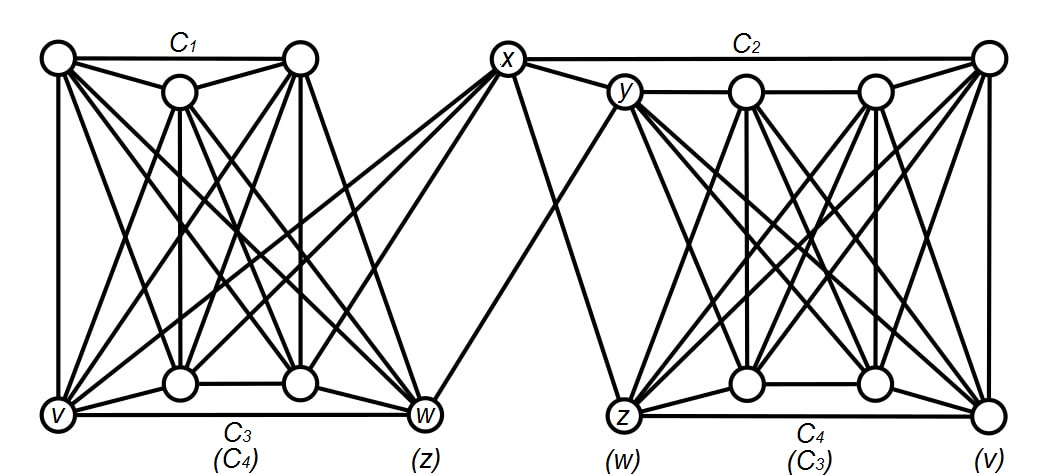,height=175pt,width=400pt}
\caption{Graph $G$ (and $G'$) for $k=6$}\label{G}
\end{figure}

\section{Final remarks}
The recent paper \cite{LWYL} also presents pairs of cospectral graphs where one is Hamiltonian and the other one not.
There the authors do not require regularity, but consider the generalised spectrum, 
which is the spectrum of the adjacency matrix together with the spectrum of the adjacency matrix of the complement.
Our examples in Section~\ref{c} are regular, and therefore also cospectral with respect to the adjacency matrix of the 
complement, and as mentioned earlier, with respect to many other types of matrices including the Laplacian $L$ and the 
signless Laplacian $Q$.

Being a disjoint union of complete graphs is not determined by the spectrum of $L$ or $Q$.
So we cannot copy the proof of Theorem~\ref{P} for these matrices.
However, we can modify the proof using the fact that a disjoint union of paths is determined by the Laplacian spectrum
as well as the signless Laplacian spectrum.
So we can conclude that Theorem~\ref{P} also holds if we consider $L$ or $Q$.

Many NP-hard graph properties are known, and there is no point in trying each of them, not even if we restrict to regular 
graphs.
There is, however, one NP-hard graph property for which it will be difficult to decide if it is characterized by the spectrum.
This is the problem of being $k$-regular with chromatic index (edge-chromatic number) equal to $k$.
(By Vizing's theorem the chromatic index is $k$ or $k+1$.)
The problem can be rephrased as:
{\lq}Does a given $k$-regular graph $G$ admit a partition of the edge set into perfect matchings?{\rq}
The answer is {\lq}no{\rq} if $G$ has odd order, and if $G$ is the Petersen graph, which is DS.
But we don't know of a cospectral pair where one graph has such a partition and the other one not.
In \cite{CGH} the chromatic index of the Latin square graphs, 
the strongly regular graphs of degree at most $18$ and their complements
has been determined; no two cospectral ones have a different chromatic index.

\end{document}